\documentclass[12pt, reqno]{amsart}

\usepackage[a4paper,margin=1.12in]{geometry}
\usepackage{amsmath,amssymb,amsthm,mathtools}
\usepackage{enumitem}
\usepackage[colorlinks=true,citecolor=blue,linkcolor=blue,urlcolor=blue]{hyperref}
\usepackage[nameinlink,capitalise,noabbrev]{cleveref}

\newtheorem{theorem}{Theorem}[section]
\newtheorem{proposition}[theorem]{Proposition}
\newtheorem{lemma}[theorem]{Lemma}
\newtheorem{corollary}[theorem]{Corollary}
\theoremstyle{definition}

\newtheorem{question}[theorem]{Question}
\newtheorem{example}[theorem]{Example}
\theoremstyle{remark}
\newtheorem{remark}[theorem]{Remark}

\DeclareMathOperator{\osc}{osc}
\DeclareMathOperator{\diam}{diam}
\DeclareMathOperator{\graph}{graph}
\newcommand{\R}{\mathbb R}

\newcommand{\G}{\mathcal G}
\newcommand{\dimH}{\dim_{\mathrm H}}

\newcommand{\udimB}{\overline{\dim}_{\mathrm B}}
\newcommand{\ldimB}{\underline{\dim}_{\mathrm B}}
\newcommand{\dimA}{\dim_{\mathrm A}}
\newcommand{\dimL}{\dim_{\mathrm L}}
\newcommand{\dimAs}[1]{\dim_{\mathrm A}^{#1}}

\title[Assouad Dimensions and Spectra of Weierstrass-Type Graphs]
{Assouad Dimensions and Spectra of Graphs of Weierstrass-Type Functions with Rapidly Growing Frequencies}

\author{Dong-Hong Xiong} \address{School of General Education, Chongqing Business Vocational College, Chongqing, 401331, P.R. China} \email{914438217@qq.com}

\author{Jun Jason Luo}
\address{College of Mathematics and Statistics,  Chongqing University, Chongqing 401331, P.R. China}
\email{jun.luo@cqu.edu.cn}


\begin{document}

\begin{abstract}
	We study the Assouad dimension, lower dimension, and Assouad spectrum
	of graphs of the functions
	\(f(x)=\sum_{n=1}^{\infty}a_n\phi(b_nx+\theta_n)\),
	where \(\phi\) is a nonconstant \(1\)-periodic \(C^2\) function,
	\(a_n,b_n>0\), \((\theta_n)\) is an arbitrary real sequence,
	\(\sum_{n=1}^{\infty}a_n<\infty\), and \(b_{n+1}/b_n\to\infty\).
	We give sufficient conditions for the graph over every nondegenerate
	compact interval to have Assouad dimension two or lower dimension one.
	For \(a_n=b_n^{-\alpha}\), \(0<\alpha<1\), an explicit frequency-gap
	condition ensures that the Assouad spectrum equals two for
	\(\alpha\leq\vartheta<1\).
	Writing \(\beta_*=\limsup_{n\to\infty}\log b_{n+1}/\log b_n\),
	we show that, if \(\beta_*>(2-\alpha)/\alpha\), the spectrum attains
	the general upper bound for \(\alpha\)-H\"older graphs whenever
	\((2-\alpha)/\beta_*\leq\vartheta<\alpha\) and \(\vartheta>0\).
	When \(\beta_*=\infty\), this determines the complete spectrum.
	Together with Bara\'nski's formulas,
	our results yield graphs whose lower, Hausdorff, upper box, and
	Assouad dimensions are four distinct numbers.
\end{abstract}
\maketitle

\section{Introduction}

The classical Weierstrass function
\[
W_{a,b}(x)=\sum_{n=0}^{\infty}a^n\cos(2\pi b^n x),
\qquad x\in\mathbb R,
\]
where \(b>1\) and \(1/b<a<1\), is a fundamental example in the theory of
continuous nowhere differentiable functions; see \cite{Hardy1916}.
The box dimension of its graph is
$2+{\log a}/{\log b}$,
and the same value was conjectured to be its Hausdorff dimension.  This
conjecture motivated a long series of investigations, including
\cite{KapParYor1984,MauldinWilliams1986,HuLau1993,
	Hunt1998,BaranskiBaranyRomanowska2014,Shen2018,RenShen2021}.
In these works, the frequencies are of the form \(b^n\), so the ratio of
two consecutive frequencies is fixed.

A different situation arises when the ratios between consecutive frequencies
tend to infinity.  Let \(\phi\in C^2(\mathbb R)\) be a nonconstant
\(1\)-periodic function, and consider
\begin{equation}\label{eq:series}
	f(x)=\sum_{n=1}^{\infty}a_n\phi(b_nx+\theta_n),
	\qquad x\in\mathbb R,
\end{equation}
where \(a_n,b_n>0\) and \((\theta_n)\) is an arbitrary real sequence. We assume that $\sum_{n=1}^{\infty}a_n<\infty$ and focus on frequency sequences satisfying ${b_{n+1}}/{b_n}\to\infty$.
The study of such series goes back to Besicovitch and
Ursell~\cite{BesicovitchUrsell1937}, who considered  the case in which
\(\phi\) is a sawtooth function and \(a_n=b_n^{-\alpha}\) for some \(0<\alpha<1\).
For a Lipschitz periodic  function \(\phi\) satisfying a suitable nondegeneracy
condition, Bara\'nski~\cite{Baranski2012} subsequently obtained exact
formulas for the Hausdorff, lower box, and upper box dimensions of the
graph.  His results show how these dimensions depend on both the decay of
the amplitudes and the growth of the frequencies.

For the classical Weierstrass function and its various generalizations,
most previous work on the geometry of the graph has concerned the Hausdorff
and box dimensions.  Much less is known about the Assouad and lower
dimensions, which measure, respectively, the largest and smallest covering complexity occurring at different
locations and scales.  In particular, the exact
Assouad dimension of the graph of \(W_{a,b}\) remains unknown.  Determining
these two dimensions for Weierstrass-type functions with rapidly growing
frequencies is one of the main aims of this paper.

The Assouad dimension has been studied for graphs associated with several
other classes of irregular functions.  Howroyd and
Yu~\cite{HowroydYu2019} proved that the graph of Brownian motion has
Assouad dimension two almost surely, while Feng and
Fraser~\cite{FengFraser2026} investigated the Assouad dimensions of graphs
of typical functions in several function spaces.  For Takagi functions,
Yu~\cite{Yu2020} obtained nontrivial lower bounds using weak tangents, and
Anttila, B\'ar\'any and K\"aenm\"aki
\cite{AnttilaBaranyKaenmaki2024} related the Assouad dimension of the
Takagi graph to the dimensions of its slices.  Jiang~\cite{Jiang2025} proved that the graphs of a family of
generalized Takagi functions, including the classical Takagi
function, have Assouad dimension one.  More recently,
Chrontsios-Garitsis~\cite{ChrontsiosGaritsis2026} obtained an upper bound
strictly below two for the graph of the classical Weierstrass function in suitable parameter ranges.
These results indicate that H\"older regularity alone does not determine
the Assouad dimension; the relation between the amplitudes and the frequency
gaps must also be taken into account.

Our first result gives two explicit criteria involving the amplitudes and
frequencies.  One implies that the graph of \(f\) in \eqref{eq:series} has lower dimension one,
while the other implies that it has Assouad dimension two. For this purpose,
define
\begin{equation}\label{eq:budgets}
	D_n=\sum_{j=1}^{n}a_jb_j,
	\qquad
	K_n=\sum_{j=1}^{n}a_jb_j^2,
	\qquad
	T_n=\sum_{j=n+1}^{\infty}a_j.
\end{equation}
These quantities bound, respectively, the first and second derivatives
of the partial sums and the size of the remaining tail of the series in  \eqref{eq:series}. For a nondegenerate compact
interval $I\subset\R$, write the graph of $f$ over $I$ by
\[
\G_f(I)=\graph(f|_I)
=\{(x,f(x)):x\in I\}.
\]

\begin{theorem}\label{thm:block}
Let \(I\subset\mathbb R\) be a nondegenerate compact interval, and let $f$ be as in \eqref{eq:series}.  Assume that
$b_n\to\infty$ and $\sum_{n=1}^\infty a_n<\infty$.
\begin{enumerate}[label=\textup{(\roman*)}]
\item\label{item:lower-general}
If there is a strictly increasing  sequence of indices $(n_k)$ such that
\begin{equation}\label{eq:lower-criterion}
 \frac{D_{n_k-1}}{a_{n_k}b_{n_k}}\to 0,
 \quad
 b_{n_k}T_{n_k}\to 0
 \qquad(k\to\infty),
\end{equation}
then the lower dimension $\dimL\G_f(I)=1$.

\item\label{item:assouad-general}
If there is a strictly increasing sequence of indices \((n_k)\)
such that 
\begin{equation}\label{eq:assouad-criterion}
	\begin{gathered}
		a_{n_k+1}b_{n_k+1}\to\infty,\\[1ex]
		\frac{D_{n_k-1}}{a_{n_k}b_{n_k}}
		+\frac{T_{n_k+1}}{a_{n_k+1}}
		+\frac{D_{n_k}}{a_{n_k+1}b_{n_k+1}}
		\to 0 	\qquad(k\to\infty),\\[1ex]
		\sup_{k\geq1}K_{n_k}a_{n_k+1}<\infty.
	\end{gathered}
\end{equation}
then the Assouad dimension \(\dimA\G_f(I)=2\).
\end{enumerate}
\end{theorem}

The two conclusions use different pairs of scales.  For the lower dimension,
\eqref{eq:lower-criterion} ensures that the \(n\)th summand dominates the
preceding partial sum on an interval of length comparable to \(b_n^{-1}\),
while the remaining tail is negligible at this scale.  The graph is
therefore close to a rectifiable arc, which yields lower dimension one.
For the Assouad dimension, \eqref{eq:assouad-criterion} provides a point
near which the variation of the partial sum through the \(n\)th term
is controlled by the next amplitude.
On a small interval around this point, the next summand completes many
oscillations whose vertical size is comparable to the length of the
interval.  This produces the quadratic covering growth required for
Assouad dimension two.  

When \(a_n=b_n^{-\alpha}\), these criteria become
quantitative conditions on the gaps between consecutive frequencies.

\begin{theorem}\label{thm:power}
Let \(I\subset\mathbb R\) be a nondegenerate compact interval. Let \(\phi\in C^2(\mathbb R)\) be a nonconstant
\(1\)-periodic function,  $0<\alpha<1$.  Define
\begin{equation}\label{eq:power-series}
 f_\alpha(x)=\sum_{n=1}^{\infty}b_n^{-\alpha}
 \phi(b_nx+\theta_n), \qquad x\in {\mathbb R},
\end{equation} where $b_n>0$ and $\theta_n\in {\mathbb R}$.
Assume that $b_{n+1}/b_n\to\infty$.  Then $f_\alpha$ is globally  $\alpha$-H\"older continuous.  Moreover:
\begin{enumerate}[label=\textup{(\roman*)}]
\item If $\liminf\limits_{n\to\infty}{b_n}/{b_{n+1}^{\alpha}}=0$,
then  $\dimL\G_{f_\alpha}(I)=1$.

\item If  $\liminf\limits_{n\to\infty} {b_n^{2-\alpha}}/{b_{n+1}^{\alpha}}<\infty$,
then $\dimA\G_{f_\alpha}(I)=2, \  \dimL\G_{f_\alpha}(I)=1$, and the Assouad spectrum $\dimAs{\vartheta}\G_{f_\alpha}(I)=2$ for every $\alpha\leq\vartheta<1$.
\end{enumerate}
\end{theorem}

The value $\vartheta=\alpha$ is determined by the amplitude and period of
each summand.  The $(n+1)$st summand has amplitude
$b_{n+1}^{-\alpha}$ and period $b_{n+1}^{-1}$, and these are precisely
the two scales related by $r=R^{1/\vartheta}$ when $\vartheta=\alpha$.
Under the gap condition in \cref{thm:power}\textup{(ii)}, the spectrum
therefore attains its maximal value from $\vartheta=\alpha$ onward.
Below this value, the general estimate
for graphs of $\alpha$-H\"older functions gives
\[
\dimAs{\vartheta}\G_{f_\alpha}(I)
\leq\frac{2-\alpha-\vartheta}{1-\vartheta}<2.
\]
See \cite[Theorem~1.1]{ChrontsiosGaritsisTyson2026}, which gives the
same bound for the regularized Assouad spectrum and hence for the ordinary
spectrum used here. The next theorem identifies an interval below
$\alpha$ on which this bound is attained, using only the limit superior
of the logarithmic frequency ratios.

\begin{theorem}\label{thm:subcritical-spectrum}
Under the assumptions of \cref{thm:power}, let
\[
 \beta_*:=\limsup_{n\to\infty}
 \frac{\log b_{n+1}}{\log b_n}.
\]
If $\beta_*>(2-\alpha)/\alpha$, then
\[
 \dim_A^\vartheta\G_{f_\alpha}(I)
 =\min\left\{2,\frac{2-\alpha-\vartheta}{1-\vartheta}\right\}
\]
for every $0<\vartheta<1$ satisfying
$\vartheta\geq(2-\alpha)/\beta_*$, with the convention $1/\infty=0$.
\end{theorem}

If $\beta_*=\infty$, the restriction
$\vartheta\geq(2-\alpha)/\beta_*$ is automatic. Thus
\cref{thm:subcritical-spectrum} determines the complete Assouad
spectrum in this case. Chrontsios-Garitsis and Tyson
\cite[Section~3]{ChrontsiosGaritsisTyson2026} constructed graphs
attaining the general H\"older bound by modifying a given graph.
Our argument uses the same principle of combining local containment
with oscillation estimates, but establishes the bound directly for
the series \eqref{eq:power-series}, for every phase sequence.

To compare these results with the Hausdorff and box dimensions, we recall
the following specialization of Bara\'nski's formulas.  The nondegeneracy
assumption in his theorem is satisfied here, since every nonconstant
\(C^1\) periodic function has an interval on which its derivative has
constant sign and is bounded away from zero.

\begin{theorem}[\cite{Baranski2012}]\label{thm:global}
Under the assumptions of \cref{thm:power}, put
	\[
	\beta_*=\limsup_{n\to\infty}
	\frac{\log b_{n+1}}{\log b_n}\in[1,\infty].
	\]
	Then we have 
	\begin{equation*}
		\dimH\G_{f_\alpha}(I)=\ldimB\G_{f_\alpha}(I)
		=1+\frac{1-\alpha}{1-\alpha+\alpha\beta_*},
		\quad
		\udimB\G_{f_\alpha}(I)=2-\alpha,
	\end{equation*}
	where the fraction is interpreted as zero when $\beta_*=\infty$.
\end{theorem}

 Combining \cref{thm:power} with  \cref{thm:global} yields the following separation of dimensions.

\begin{corollary}\label{cor:four}
Under the assumptions of \cref{thm:power}, suppose that
\[
\beta_*:=\limsup_{n\to\infty}\frac{\log b_{n+1}}{\log b_n}<\infty
\quad\text{and}\quad
\liminf_{n\to\infty}\frac{b_n^{2-\alpha}}{b_{n+1}^{\alpha}}<\infty.
\]
Then
\begin{align*} 
 \dimL\G_{f_\alpha}(I)=1
 &<\dimH\G_{f_\alpha}(I)=\ldimB\G_{f_\alpha}(I) \\
 &=1+\frac{1-\alpha}{1-\alpha+\alpha\beta_*}
 <\udimB\G_{f_\alpha}(I)=2-\alpha
 <\dimA\G_{f_\alpha}(I)=2.
\end{align*}
Moreover, the Assouad spectrum 
\(\dimAs{\vartheta}\G_{f_\alpha}(I)=2\) for \(\alpha\leq\vartheta<1\).
\end{corollary}

The rest of this paper is organized as follows. \Cref{sec:oscillation} states the definitions of Assouad and lower dimensions and proves \cref{thm:block}, while
\cref{sec:local} proves \cref{thm:power,thm:subcritical-spectrum}
and their corollaries. Examples, final remarks and open problems
are given in \cref{sec:questions}.

\section{Proof of Theorem 1.1}
\label{sec:oscillation}

We write \(B(z,R)\) for the closed Euclidean ball and
\(\osc_I f=\sup_I f-\inf_I f\). The notation \(A_n\asymp B_n\) means that
\(cB_n\leq A_n\leq CB_n\) for positive constants \(c,C\)
independent of \(n\). For a nonempty bounded set \(E\subset\mathbb R^2\), let \(N_r(E)\)
denote the least number of sets of diameter at most \(r\) required to
cover \(E\).  The Assouad dimension of \(E\) is
\begin{align*}
	\dimA E=\inf\biggl\{s\geq0:\;&
	\text{there exist }C,\rho>0\text{ so that }
	N_r\bigl(B(z,R)\cap E\bigr)
	\leq C\left(\frac{R}{r}\right)^s\\
	&\text{for every }z\in E
	\text{ and }0<r<R<\rho\biggr\}.
\end{align*}
The lower dimension of \(E\) is
\begin{align*}
	\dimL E=\sup\biggl\{s\geq0:\;&
	\text{there exist }c,\rho>0\text{ so that }
	N_r\bigl(B(z,R)\cap E\bigr)
	\geq c\left(\frac{R}{r}\right)^s\\
	&\text{for every }z\in E
	\text{ and }0<r<R<\rho\biggr\}.
\end{align*}

Hausdorff and box dimensions describe the scaling of a set as a
whole.  The Assouad and lower dimensions instead measure,
respectively, the largest and smallest covering complexity of
$B(z,R)\cap E$ at a finer scale $r$, uniformly over all locations
and scales.  Thus a sequence with large covering numbers can provide a lower
bound for \(\dimA E\), while a sequence with small covering numbers can
provide an upper bound for \(\dimL E\).  See
\cite{Assouad1983,Fraser2014,Fraser2021} for further background.

The Assouad spectrum, introduced by Fraser and
Yu~\cite{FraserYu2018}, restricts the two scales by
\(r=R^{1/\vartheta}\).  For \(0<\vartheta<1\), it is defined by
\begin{align*}
	\dimAs{\vartheta}E=\inf\biggl\{s\geq0:\;&
	\text{there exist }C,\rho>0\text{ so that }
	N_{R^{1/\vartheta}}\bigl(B(z,R)\cap E\bigr)
	\leq
	C\left(\frac{R}{R^{1/\vartheta}}\right)^s\notag\\
	&\text{for every }z\in E
	\text{ and }0<R<\rho\biggr\}.
\end{align*}
The Assouad spectrum measures the same covering behaviour under the
prescribed relation \(r=R^{1/\vartheta}\); see
\cite{FraserEtAl2019}.  

We will use the following direct consequence of these definitions.
Suppose that $z_n\in E$, $0<r_n<R_n\to0$, $R_n/r_n\to\infty$, and
\[
 N_{r_n}(B(z_n,R_n)\cap E)\geq c(R_n/r_n)^s
\]
for some $c>0$ independent of $n$. Then $\dimA E\geq s$.
Indeed, a uniform upper estimate with any exponent $t<s$ would imply
$c(R_n/r_n)^{s-t}\leq C$, a contradiction. If in addition
$r_n=R_n^{1/\vartheta}$, the same argument gives
$\dimAs{\vartheta}E\geq s$.

Set
\begin{equation*} 
 P_n(x)=\sum_{j=1}^{n}a_j\phi(b_jx+\theta_j),
 \qquad H_n(x)=f(x)-P_n(x).
\end{equation*}
Since \(\phi\in C^2(\mathbb R)\) is a nonconstant
\(1\)-periodic function, we use the constants
\[
 M_0=\|\phi\|_\infty,
 \quad M_1=\|\phi'\|_\infty,
 \quad M_2=\|\phi''\|_\infty,
 \quad \kappa=\max \phi- \min \phi>0.
\]
By \eqref{eq:budgets}, we have  
\begin{equation}\label{eq:basic-norms}
 \|P_n'\|_\infty\leq M_1D_n,
 \quad
 \|P_n''\|_\infty\leq M_2K_n,
 \quad
 \|H_n\|_\infty\leq M_0T_n.
\end{equation}

\begin{lemma}\label{lem:three-block}
Let $I\subset\R$ be an interval of length $\ell$.  For every $n\geq0$,
\begin{equation}\label{eq:osc-upper}
 \osc_I f\leq M_1D_n\ell+2M_0T_n,
\end{equation}
where $D_0=0$.  If $I$ contains a full period of
$x\mapsto\phi(b_{n+1}x+\theta_{n+1})$, then
\begin{equation}\label{eq:osc-lower}
 \osc_I f\geq \kappa a_{n+1}
 -\frac{M_1D_n}{b_{n+1}}-2M_0T_{n+1}.
\end{equation}
\end{lemma}

\begin{proof}
For $x,y\in I$, the mean value theorem and \eqref{eq:basic-norms} give
\[
 |P_n(x)-P_n(y)|\leq M_1D_n\ell,
 \qquad |H_n(x)-H_n(y)|\leq2M_0T_n,
\]
which proves \eqref{eq:osc-upper}.

Choose $s_+,s_-\in[0,1]$ with
$\phi(s_+)=\max\phi$ and $\phi(s_-)=\min\phi$.  A full period of
$x\mapsto\phi(b_{n+1}x+\theta_{n+1})$ inside $I$
contains points $u,v$ satisfying
\[
 b_{n+1}u+\theta_{n+1}\equiv s_+\pmod1,
 \qquad
 b_{n+1}v+\theta_{n+1}\equiv s_-\pmod1.
\]
They may be chosen in the same period, so $|u-v|\leq b_{n+1}^{-1}$.
Writing 
$$f(x)=P_n(x)+a_{n+1}\phi(b_{n+1}x+\theta_{n+1})+H_{n+1}(x),$$
we obtain
\begin{align*}
 |f(u)-f(v)|
 &\geq \kappa a_{n+1}-|P_n(u)-P_n(v)|
      -|H_{n+1}(u)-H_{n+1}(v)|\\
 &\geq \kappa a_{n+1}-\frac{M_1D_n}{b_{n+1}}
      -2M_0T_{n+1}.
\end{align*}
Taking the supremum over pairs in $I$ proves \eqref{eq:osc-lower}.
\end{proof}

\bigskip

\begin{proof}[Proof of \cref{thm:block}]
	We first prove part~\ref{item:lower-general}.  Choose a strictly
	increasing  sequence of indices $(n_k)_{k\geq1}$ along which both limits in
	\eqref{eq:lower-criterion} hold.  For each $k$, set $n=n_k$.  All
	limits in this part of the proof are taken as $k\to\infty$.
	
	Since $\phi$ is nonconstant, $\phi'$ is nonzero at some point.  By
	continuity, there exist a closed interval $J\subset(0,1)$ and a
	constant $\eta>0$ such that $\phi'$ has a constant sign on $J$ and
	\(|\phi'(t)|\geq\eta\) for \(t\in J\).
	Choose $t_0\in\operatorname{int}J$ and $c_0>0$ sufficiently small so that
	\([t_0-3c_0,t_0+3c_0]\subset J\).
	
	Let $I=[u,v]$ be an interval with $u<v$, and define the middle half of $I$ by
	\[
	I_0=\left[u+\frac{v-u}{4},\,v-\frac{v-u}{4}\right].
	\]
	The solutions of
	\(
	b_nx+\theta_n\equiv t_0\pmod 1
	\)
	form an arithmetic progression with spacing $1/b_n$.  Since
	$b_n\to\infty$, for all sufficiently large $k$ at least one such
	solution belongs to $I_0$.  Choose one and denote it by $x_n$.
	
	Set
	\[
	J_n=\left[x_n-\frac{2c_0}{b_n},
	x_n+\frac{2c_0}{b_n}\right].
	\]
	Because $x_n\in I_0$ and $2c_0/b_n\to0$, we have $J_n\subset I$ for
	all sufficiently large $k$.  Moreover, if $x\in J_n$, then, using
	the representative of $b_nx+\theta_n$ near $t_0$,
	\[
	b_nx+\theta_n
	\equiv t_0+b_n(x-x_n)
	\pmod 1
	\]
	and
	\[
	t_0+b_n(x-x_n)
	\in[t_0-2c_0,t_0+2c_0]\subset J.
	\]
	It follows that the derivative of the $n$th summand \(a_n\phi(b_nx+\theta_n)\)	has a constant sign on $J_n$ and satisfies
	\[
	\bigl|a_nb_n\phi'(b_nx+\theta_n)\bigr|
	\geq\eta a_nb_n
	\qquad(x\in J_n).
	\]
	
	By \eqref{eq:basic-norms},
	\(|P_{n-1}'(x)|\leq M_1D_{n-1}\) for \(x\in\mathbb R\).
	The first limit in \eqref{eq:lower-criterion} gives
	\(M_1D_{n-1}=o(a_nb_n)\). From $P_n(x)=P_{n-1}(x)+a_n\phi(b_nx+\theta_n)$, it follows that, for sufficiently large $k$, the derivative $P_n'$ has
	the same sign as the derivative of the $n$th summand throughout $J_n$, and
	\begin{equation*} 
		|P_n'(x)|
		\geq \eta a_nb_n-M_1D_{n-1}
		\geq\frac{\eta}{2}a_nb_n
		\qquad(x\in J_n).
	\end{equation*}
	In particular, $P_n$ is strictly monotone on $J_n$.
	
	We now choose the two covering scales.  Let
	\[
	R_n=\frac{c_0}{b_n},
	\qquad
	r_n=4M_0T_n,
	\qquad
	z_n=(x_n,f(x_n)).
	\]
	The second limit in \eqref{eq:lower-criterion} implies that
	\begin{equation*}
		\frac{r_n}{R_n}
		=\frac{4M_0}{c_0}b_nT_n
		\longrightarrow0.
	\end{equation*}
	Thus $0<r_n<R_n$ for sufficiently large $k$.
	
	Suppose that $(x,f(x))\in B(z_n,R_n)\cap\G_f(I)$. Then $|x-x_n|\leq R_n={c_0}/{b_n}$. 	Consequently, $x\in J_n$.  Furthermore, by \eqref{eq:basic-norms}, we have
	\begin{align}
		|P_n(x)-P_n(x_n)|
		\leq |f(x)-f(x_n)|
		+|H_n(x)|+|H_n(x_n)|	\leq R_n+2M_0T_n.
		\label{eq:vertical-pn}
	\end{align}
	
	Define
	\[
	E_n=
	\left\{x\in J_n:
	|P_n(x)-P_n(x_n)|
	\leq R_n+2M_0T_n\right\}.
	\]
	Since $P_n$ is continuous and strictly monotone on $J_n$, the set
	$E_n$ is an interval.  Let $A_n=\graph(P_n|_{E_n})$. The length of the graph of a monotone $C^1$ function is bounded by
	the sum of its horizontal and vertical variations.  Hence
	\begin{align*}
		 \operatorname{length}(A_n)  \leq \diam E_n +\int_{E_n}|P_n'(x)|\,dx  \leq 4R_n+2(R_n+2M_0T_n)\leq 6R_n+4M_0T_n.
	\end{align*}
	Since $M_0T_n=r_n/4=o(R_n)$, there is a constant $C_1$, independent
	of $n$, such that
	\begin{equation}\label{eq:arc-length}
	 \operatorname{length}(A_n)\leq C_1R_n
	\end{equation}
	for sufficiently large $k$.
	
	For every $(x,f(x))\in B(z_n,R_n)\cap\G_f(I)$, the point
	$(x,P_n(x))$ belongs to $A_n$ by \eqref{eq:vertical-pn}, and
	\[
	|f(x)-P_n(x)|=|H_n(x)|
	\leq M_0T_n=\frac{r_n}{4}.
	\]
	Therefore
	\[
	B(z_n,R_n)\cap\G_f(I)
	\subset (A_n)_{r_n/4},
	\]
	where $(A_n)_{r_n/4}$ denotes the Euclidean $r_n/4$-neighbourhood
	of $A_n$.
	
	Parametrize $A_n$ by arclength and choose a finite sequence of
	points on $A_n$, including its endpoints, so that consecutive
	points are at arclength distance at most $r_n/4$.  The number of
	chosen points is at most
	\(2+{4\operatorname{length}(A_n)}/{r_n}\).
	Every point of $A_n$ lies within Euclidean distance $r_n/4$ of one
	of these points.  It follows that every point of $(A_n)_{r_n/4}$
	lies within distance $r_n/2$ of one of them.  Balls of radius
	$r_n/2$ have diameter $r_n$, and hence \eqref{eq:arc-length} gives
	\begin{equation}\label{eq:lower-witness}
		N_{r_n}\bigl(B(z_n,R_n)\cap\G_f(I)\bigr)
		\leq 2+\frac{4C_1R_n}{r_n}
		\leq C_2\frac{R_n}{r_n}
	\end{equation}
	for a constant $C_2$ independent of $n$.  In the last inequality
	we used $R_n/r_n\to\infty$.
	
	Suppose, towards a contradiction, that
	$\dimL\G_f(I)>1$.  Choose
	\(1<s<\dimL\G_f(I)\).
	By the definition of the lower dimension, there exist $c,\rho>0$
	such that
	\[
	N_r(B(z,R)\cap\G_f(I))
	\geq c\left(\frac{R}{r}\right)^s
	\]
	for every $z\in\G_f(I)$ and all $0<r<R<\rho$.  Since $R_n\to0$
	and $r_n/R_n\to0$, this estimate applies to
	$(z_n,R_n,r_n)$ for all sufficiently large $k$.  Combining it with
	\eqref{eq:lower-witness}, we obtain
	\[
	c\left(\frac{R_n}{r_n}\right)^{s-1}\leq C_2.
	\]
	This is impossible because $s>1$ and $R_n/r_n\to\infty$.  Thus
	\(\dimL\G_f(I)\leq1\).
	
	For completeness, we verify the reverse inequality.  The graph
	$\G_f(I)$ is a compact connected set containing more than one
	point.  Let
	\[
	E=\G_f(I),\qquad 0<R<\frac{\diam E}{4},
	\]
	and fix $z\in E$.  There exists $y\in E$ such that $d(z,y)>R$.
	Since $E$ is connected and the function $w\mapsto d(z,w)$ is
	continuous, its image is an interval containing $0$ and a number
	larger than $R$.  It therefore contains $[0,R]$.  If $0<r<R/4$,
	choose points $w_j\in E$ satisfying
	\[
	d(z,w_j)=(2j-1)r,
	\qquad
	1\leq j\leq\left\lfloor\frac{R}{2r}\right\rfloor.
	\]
	All these points belong to $B(z,R)\cap E$, and for $i\neq j$,
	the reverse triangle inequality gives
	\[
	d(w_i,w_j)
	\geq\bigl|d(z,w_i)-d(z,w_j)\bigr|
	\geq2r.
	\]
	A set of diameter at most $r$ can therefore contain at most one
	of the points $w_j$.  Hence
	\[
	N_r(B(z,R)\cap E)
	\geq\left\lfloor\frac{R}{2r}\right\rfloor
	\geq\frac{R}{4r}
	\qquad(0<r<R/4).
	\]
	If $R/4\leq r<R$, then
	\[
	N_r(B(z,R)\cap E)\geq1\geq\frac{R}{4r}.
	\]
	Consequently,
	\(N_r(B(z,R)\cap E)\geq {R}/{4r}\) for \(0<r<R<\diam E/4\), 
	and therefore $\dimL E\geq1$.  We conclude that
	\(\dimL\G_f(I)=1\).
	
	We next prove part~\ref{item:assouad-general}.  Choose a strictly
	increasing sequence of indices $(n_k)_{k\geq1}$ along which all the conditions in
	\eqref{eq:assouad-criterion} hold simultaneously, and again set
	$n=n_k$.
	
	By assumption,  $\phi'$ is not identically zero and $\int_0^1\phi'(t)dt=\phi(1)-\phi(0)=0$.  We may therefore choose
	$t_+,t_-\in[0,1)$ and $\eta>0$ such that
	\[
	\phi'(t_+)\geq2\eta,
	\qquad
	\phi'(t_-)\leq-2\eta.
	\]
	
	Let
	\[
	I_1=
	\left[u+\frac{3(v-u)}8,\,
	v-\frac{3(v-u)}8\right],
	\]
	which is a closed interval contained in the interior of $I_0$.
	As above, the solutions of $b_nx+\theta_n\equiv t_+ \pmod1$,
	have spacing $1/b_n$.  Hence, for sufficiently large $k$, we
	may choose a solution $x_{n,+}\in I_1$ and an integer \(m_n\) such that $b_nx_{n,+}+\theta_n=t_++m_n$.
	
	Define
	\[
	x_{n,-}=\frac{t_-+m_n-\theta_n}{b_n}.
	\]
	It follows that $b_nx_{n,-}+\theta_n\equiv t_-\pmod 1$	and
	\[
	|x_{n,-}-x_{n,+}|
	=\frac{|t_--t_+|}{b_n}
	<\frac1{b_n}.
	\]
	Since $x_{n,+}\in I_1$, $I_1$ has positive distance from
	 the endpoints of $I_0$, and $b_n^{-1}\to0$, both $x_{n,+}$ and $x_{n,-}$
	belong to $I_0$ for all sufficiently large $k$.
	
	Using \eqref{eq:basic-norms}, we obtain
	\begin{align*}
		P_n'(x_{n,+})
		&=P_{n-1}'(x_{n,+})
		+a_nb_n\phi'(t_+)
		\geq-M_1D_{n-1}+2\eta a_nb_n, \\
		P_n'(x_{n,-})
		&=
		P_{n-1}'(x_{n,-})
		+a_nb_n\phi'(t_-)\leq M_1D_{n-1}-2\eta a_nb_n.
	\end{align*}
	The condition
	\({D_{n-1}}/{a_nb_n}\to 0\) in \eqref{eq:assouad-criterion} shows that
	\[
	P_n'(x_{n,+})>0,
	\qquad
	P_n'(x_{n,-})<0
	\]
	for all sufficiently large $k$.  Since $P_n'$ is continuous, the
	intermediate value theorem gives a point $c_n$ between
	$x_{n,+}$ and $x_{n,-}$ such that 	\(P_n'(c_n)=0\).
	In particular, $c_n\in I_0$.
	
	Set
	\(R_n=a_{n+1}, \;  r_n=1/{b_{n+1}}\),
	and let
	\[
	I_n=\left[c_n-\frac{R_n}{2},\,
	c_n+\frac{R_n}{2}\right].
	\]
	Since $a_{n+1}\to0$ and $I_0$ has positive distance from the endpoints of $I$, we have
	$I_n\subset I$ for all sufficiently large $k$.
	
	For $x\in I_n$, Taylor's theorem, $P_n'(c_n)=0$, and
	\eqref{eq:basic-norms} give
	\begin{align*}
		|P_n(x)-P_n(c_n)|\leq\frac12\|P_n''\|_\infty|x-c_n|^2
		\leq\frac{M_2K_n}{2}\left(\frac{R_n}{2}\right)^2
		=\frac{M_2}{8}K_nR_n^2.
	\end{align*}
    By \eqref{eq:assouad-criterion}, there is $B<\infty$, independent
    of the selected index $n$, such that $K_nR_n=K_na_{n+1}\leq B$.
    Hence
	\[
	 \sup_{x\in I_n}|P_n(x)-P_n(c_n)|
     \leq\frac{M_2B}{8}R_n.
	\]
	It follows that
	\begin{align*}
		|f(x)-f(c_n)|
		&\leq |P_n(x)-P_n(c_n)|	+a_{n+1}\bigl|
		\phi(b_{n+1}x+\theta_{n+1})
		-\phi(b_{n+1}c_n+\theta_{n+1})
		\bigr|\\
		&\quad+|H_{n+1}(x)|+|H_{n+1}(c_n)|\\
		&\leq\left(\frac{M_2B}{8}+2M_0\right)R_n+2M_0T_{n+1}.
	\end{align*}
	The condition $T_{n+1}/a_{n+1}\to0$ in \eqref{eq:assouad-criterion} yields
	$T_{n+1}=o(R_n)$.  Hence there is a constant $C_0>1$,
	independent of $n$, such that
	\[
	|f(x)-f(c_n)|\leq C_0R_n
	\qquad(x\in I_n)
	\]
	for all sufficiently large $k$.  Since also
	$|x-c_n|\leq R_n/2$, enlarging $C_0$ if necessary gives
	\begin{equation}\label{eq:local-containment}
		\graph(f|_{I_n})
		\subset B(z_n,C_0R_n),
		\qquad
		z_n=(c_n,f(c_n)).
	\end{equation}
	
	The interval $I_n$ has length $R_n$, whereas one period of the
	$(n+1)$st summand has length $r_n=1/b_{n+1}$.  Thus $I_n$
	contains at least
	\(
	\lfloor{R_n}/{r_n}\rfloor-2
	\)
	complete period intervals.  Moreover,
	\(
	{R_n}/{r_n}=a_{n+1}b_{n+1}\to\infty
	\)
	by \eqref{eq:assouad-criterion}.  Retain every fourth complete
	period interval and denote the resulting family by $\mathcal J_n$.
	There is a constant $c_1>0$, independent of $n$, such that
	\begin{equation}\label{eq:number-periods}
		\#\mathcal J_n
		\geq c_1\frac{R_n}{r_n}
	\end{equation}
	for all sufficiently large $k$.  The horizontal distance between
	any two distinct intervals in $\mathcal J_n$ is at least $2r_n$.
	
	For each $J\in\mathcal J_n$, \cref{lem:three-block} gives
	\[
	\osc_J f
	\geq
	\kappa a_{n+1}
	-\frac{M_1D_n}{b_{n+1}}
	-2M_0T_{n+1}.
	\]
	Using 	\eqref{eq:assouad-criterion}, we obtain
	\begin{equation}\label{eq:period-oscillation}
		\osc_J f\geq\frac{\kappa}{2}R_n
		\qquad(J\in\mathcal J_n)
	\end{equation}
	for all sufficiently large $k$.
	
	For a fixed $J\in\mathcal J_n$, since \(f\) is continuous, \(\graph(f|_J)\) is connected.  Its projection onto the vertical axis is an interval of length $\osc_Jf$.   It follows from
	\eqref{eq:period-oscillation} that at least \({\kappa R_n}/{2r_n}\)
	sets of diameter at most $r_n$ are required to cover
	$\graph(f|_J)$.  Moreover, because distinct retained intervals are
	horizontally separated by at least $2r_n$, a set of diameter at
	most $r_n$ cannot meet the graphs above two different intervals in
	$\mathcal J_n$.  Summing over all $J\in\mathcal J_n$ and using
	\eqref{eq:number-periods}, we obtain
	\[
	N_{r_n}\bigl(\graph(f|_{I_n})\bigr)
	\geq
	c_2\left(\frac{R_n}{r_n}\right)^2
	\]
	for some constant $c_2>0$. Therefore, by \eqref{eq:local-containment},  it follows that 
	\begin{equation}\label{eq:quadratic}
		N_{r_n}\bigl(B(z_n,C_0R_n)\cap\G_f(I)\bigr)
		\geq
		c_2\left(\frac{R_n}{r_n}\right)^2.
	\end{equation}
	
    Since $C_0R_n\to0$ and $C_0R_n/r_n\to\infty$,
    \eqref{eq:quadratic} gives quadratic covering growth with outer
    radius $C_0R_n$ and constant $c_2/C_0^2$. By the observation above,
    $\dimA\G_f(I)\geq2$. The reverse inequality follows from
    $\G_f(I)\subset\mathbb R^2$, proving part~\textup{(ii)}.
\end{proof}

\section{Power amplitudes and the Assouad spectrum}
\label{sec:local}

\begin{lemma}\label{lem:dominance}
Let $0<\alpha<1$, $a_n=b_n^{-\alpha}$, and
$b_{n+1}/b_n\to\infty$.  Then
\begin{align}\label{eq:dominance}
 D_n&=(1+o(1))b_n^{1-\alpha}, \qquad  K_n =(1+o(1))b_n^{2-\alpha},\\
 T_n&=(1+o(1))b_{n+1}^{-\alpha}.\label{eq:tail-dominance}
\end{align}
\end{lemma}

\begin{proof}
Fix $s>0$.  Given $\varepsilon\in(0,1)$, choose $n_0$ so that
$b_j/b_{j+1}\leq\varepsilon$ for $j\geq n_0$.  For $n-k\geq n_0$,
\[
 \left(\frac{b_{n-k}}{b_n}\right)^s\leq\varepsilon^{ks}.
\]
Consequently,
\[
 \sum_{j=n_0}^{n-1}\left(\frac{b_j}{b_n}\right)^s
 \leq\frac{\varepsilon^s}{1-\varepsilon^s}.
\]
The finite sum with $j<n_0$, divided by $b_n^s$, tends to zero.  Letting
$\varepsilon\downarrow0$ shows
$\sum_{j=1}^n b_j^s=(1+o(1))b_n^s$.  Taking
$s=1-\alpha$ and $s=2-\alpha$ proves \eqref{eq:dominance}.

Similarly, for $k\geq1$ and large $n$,
\[
 \frac{a_{n+1+k}}{a_{n+1}}
 =\left(\frac{b_{n+1}}{b_{n+1+k}}\right)^\alpha
 \leq\varepsilon^{k\alpha}.
\]
Hence
\[
 1\leq\frac{T_n}{a_{n+1}}
 \leq1+\frac{\varepsilon^\alpha}{1-\varepsilon^\alpha},
\]
which proves \eqref{eq:tail-dominance} after
$\varepsilon\downarrow0$.
\end{proof}

The following equivalence concerns the hypotheses of \cref{thm:block}; it does
not assert that the gap conditions are necessary for the dimension values.

\begin{proposition}
\label{prop:exact-reduction}
Assume $0<\alpha<1$, $a_n=b_n^{-\alpha}$, and
$b_{n+1}/b_n\to\infty$.
\begin{enumerate}[label=\textup{(\roman*)}]
\item There exists a strictly increasing  sequence of indices $(n_k)$ for which both
limits in \eqref{eq:lower-criterion} hold if and only if  $\liminf\limits_{n\to\infty}{b_n}/{b_{n+1}^{\alpha}}=0$.
\item There exists a strictly increasing sequence of indices $(n_k)$
for which all the conditions in \eqref{eq:assouad-criterion} hold
if and only if
$\liminf\limits_{n\to\infty}{b_n^{2-\alpha}}/{b_{n+1}^{\alpha}}<\infty$.
\end{enumerate}
\end{proposition}

\begin{proof}
By \cref{lem:dominance}, 
\[
 \frac{D_{n-1}}{a_nb_n}
 =(1+o(1))\left(\frac{b_{n-1}}{b_n}\right)^{1-\alpha}
 \longrightarrow0,
 \qquad
 b_nT_n=(1+o(1))\frac{b_n}{b_{n+1}^{\alpha}}.
\]
Therefore the two limits in \eqref{eq:lower-criterion} hold along one
strictly increasing sequence precisely when the ratio in part~(i) tends to
zero along that sequence.  This is equivalent to the asserted limit inferior.

For the Assouad criterion, four requirements hold along the full sequence:
\begin{align*}
 a_{n+1}b_{n+1}&=b_{n+1}^{1-\alpha}\longrightarrow\infty,\\
 \frac{D_{n-1}}{a_nb_n}
 &=(1+o(1))\left(\frac{b_{n-1}}{b_n}\right)^{1-\alpha}
 \longrightarrow0,\\
 \frac{T_{n+1}}{a_{n+1}}
 &=(1+o(1))\left(\frac{b_{n+1}}{b_{n+2}}\right)^\alpha
 \longrightarrow0,\\
 \frac{D_n}{a_{n+1}b_{n+1}}
 &=(1+o(1))\left(\frac{b_n}{b_{n+1}}\right)^{1-\alpha}
 \longrightarrow0.
\end{align*}
The only remaining requirement is boundedness of
$K_na_{n+1}=(1+o(1))b_n^{2-\alpha}/b_{n+1}^{\alpha}$ along a
strictly increasing sequence. For a nonnegative sequence, a bounded
subsequence exists if and only if its limit inferior is finite.
This proves part~(ii).
\end{proof}

The following logarithmic gap conditions imply the hypotheses above.
Part~\textup{(ii)} gives convergence to zero along a subsequence, which
is stronger than the boundedness required in \cref{prop:exact-reduction}.

\begin{proposition}\label{prop:log-gap}
Assume $0<\alpha<1$ and let $(b_n)$ be a sequence of positive real numbers
with $b_n\to\infty$.
\begin{enumerate}[label=\textup{(\roman*)}]
\item If $\limsup\limits_{n\to\infty}{\log b_{n+1}}/{\log b_n}>{1}/{\alpha}$,
then $\liminf\limits_{n\to\infty}{b_n}/{b_{n+1}^{\alpha}}=0$.
\item If $\limsup\limits_{n\to\infty}{\log b_{n+1}}/{\log b_n}> {(2-\alpha)}/{\alpha}$, 
then $\liminf\limits_{n\to\infty} {b_n^{2-\alpha}}/{b_{n+1}^{\alpha}}=0$.
\end{enumerate}
\end{proposition}

\begin{proof}
For part~(i), choose $\varepsilon>0$ and infinitely many $n$ such that
$\log b_{n+1}\geq(\alpha^{-1}+\varepsilon)\log b_n$.
Along these indices,
\[
 \log\left(\frac{b_n}{b_{n+1}^{\alpha}}\right)
 \leq-\alpha\varepsilon\log b_n\longrightarrow-\infty,
\]
which proves the first assertion. 

For part~(ii), the same argument gives infinitely many $n$ for which
\[
 \log b_{n+1}\geq
 \left(\frac{2-\alpha}{\alpha}+\varepsilon\right)\log b_n.
\]
Hence
\[
 \log\left(\frac{b_n^{2-\alpha}}{b_{n+1}^{\alpha}}\right)
 \leq-\alpha\varepsilon\log b_n\longrightarrow-\infty,
\]
proving the second assertion.
\end{proof}

\begin{remark}
	At the critical logarithmic growth rate, the ratio of logarithms
    alone does not decide even the boundedness condition.
    Indeed, let \(b_n\to\infty\), set $q={(2-\alpha)}/{\alpha}$
 	and suppose that
	\[
	b_{n+1}=b_n^qL_n, \qquad L_n>0,
	\qquad
	\frac{\log L_n}{\log b_n}\longrightarrow0.
	\]
	Then
	\[
	\frac{\log b_{n+1}}{\log b_n}
	=
	q+\frac{\log L_n}{\log b_n}
	\longrightarrow q,
	\]
	whereas
	\[
	\frac{b_n^{2-\alpha}}{b_{n+1}^{\alpha}}
	=
	L_n^{-\alpha}.
	\]
	Consequently,
    \[
     \liminf_{n\to\infty}\frac{b_n^{2-\alpha}}{b_{n+1}^{\alpha}}
     <\infty
     \quad\Longleftrightarrow\quad
     \limsup_{n\to\infty}L_n>0,
    \]
    where an infinite limit superior is allowed. Requiring the limit inferior to equal zero is equivalent to
    \(\limsup\limits_{n\to\infty}L_n=\infty\).
\end{remark}

\begin{proof}[Proof of \cref{thm:power}]
	We first prove the H\"older continuity of $f_\alpha$.  Since
	$b_{n+1}/b_n\to\infty$, there is $n_0>1$ such that $(b_n)$ is
	strictly increasing for $n\geq n_0$.  Moreover, after increasing $n_0$
	if necessary, we may assume that
	\[
	b_{n+1}\geq 2b_n
	\qquad(n\geq n_0).
	\]
	It follows that $b_n\to\infty$ and
	\(\sum_{n=1}^{\infty}b_n^{-\alpha}<\infty\).
	Thus the series defining $f_\alpha$ converges uniformly on $\mathbb R$.
	
	Let $x,y\in\mathbb R$ and write $h=|x-y|$.  Suppose first that
	\(0<h\leq b_{n_0}^{-1}\).
	Since $(b_n)_{n\ge n_0}$ is strictly increasing and tends to infinity, there is a
	unique $n\geq n_0$ such that
	\(b_{n+1}^{-1}<h\leq b_n^{-1}\).
	Applying \eqref{eq:osc-upper} to the interval with endpoints $x$ and
	$y$, and then using \cref{lem:dominance}, we obtain
	\begin{align*}
		|f_\alpha(x)-f_\alpha(y)|\leq M_1D_nh+2M_0T_n \leq C_0\bigl(b_n^{1-\alpha}h+b_{n+1}^{-\alpha}\bigr)\le 2C_0 h^{\alpha}
	\end{align*}
	for some constant $C_0>0$ independent of $x,y$ and $n$.
	
    If $h>b_{n_0}^{-1}$, then uniform convergence gives
	\[
	\|f_\alpha\|_\infty
	\leq M_0\sum_{n=1}^{\infty}b_n^{-\alpha}
	<\infty,
	\]
    and $|f_\alpha(x)-f_\alpha(y)|
		\leq 2\|f_\alpha\|_\infty \leq
		2\|f_\alpha\|_\infty b_{n_0}^{\alpha}h^\alpha$.    Consequently, for any $x,y\in \mathbb R$, we have
	$|f_\alpha(x)-f_\alpha(y)|	\leq C_\alpha|x-y|^\alpha$,
	where \(C_\alpha= \max\{2C_0,\, 2\|f_\alpha\|_\infty b_{n_0}^{\alpha}\}\). 	Thus $f_\alpha$ is  $\alpha$-H\"older continuous.
	
	We now prove part~\textup{(i)}. Under the assumption, \cref{prop:exact-reduction}\textup{(i)} gives a strictly increasing
	 sequence of indices $(n_k)$ along which both limits in
	\eqref{eq:lower-criterion} hold simultaneously.  Therefore,
	\cref{thm:block}\ref{item:lower-general} applies and yields
	\(\dimL\G_{f_\alpha}(I)=1\).
		
	For part~\textup{(ii)}, \cref{prop:exact-reduction}\textup{(ii)} gives a strictly increasing sequence of indices $(n_k)$ along which all the conditions in \eqref{eq:assouad-criterion} hold. Hence
	\(\dimA\G_{f_\alpha}(I)=2\) by \cref{thm:block}\ref{item:assouad-general}.
	
	Moreover,
	\[
	\frac{b_n}{b_{n+1}^{\alpha}}
	=
	\frac{b_n^{2-\alpha}}{b_{n+1}^{\alpha}}
	\frac{1}{b_n^{1-\alpha}}.
	\]
	Along the selected subsequence the first factor on the right is
    bounded, while $b_n^{1-\alpha}\to\infty$. Thus
	\[
	\liminf_{n\to\infty}\frac{b_n}{b_{n+1}^{\alpha}}=0.
	\]
	Part~\textup{(i)} therefore gives
	\(\dimL\G_{f_\alpha}(I)=1\).

	It remains to consider the Assouad spectrum.  Fix
	\(\vartheta\in[\alpha,1)\) and use the same subsequence $(n_k)$,
    along which $b_{n_k}^{2-\alpha}/b_{n_k+1}^\alpha$ is bounded.
    For simplicity, write $n=n_k$ throughout the rest of the proof.
	
	The construction in the proof of
	\cref{thm:block}\ref{item:assouad-general} provides the scales
	\[
	R_n=a_{n+1}=b_{n+1}^{-\alpha},
	\qquad
	\lambda_n=b_{n+1}^{-1},
	\]
	an interval $I_n\subset I$ of length $R_n$, a point
	$z_n\in\G_{f_\alpha}(I)$, and a constant $C\geq1$, independent of
	$n$, such that
	\begin{equation}\label{eq:spectrum-localization}
		\graph(f_\alpha|_{I_n})
		\subset B(z_n,CR_n).
	\end{equation}
	The same construction also shows that every interval
	$J\subset I_n$ corresponding to a complete period of the
	$(n+1)$st summand satisfies
	\begin{equation}\label{eq:spectrum-period-oscillation}
		\osc_J f_\alpha\geq\frac{\kappa}{2}R_n
	\end{equation}
	for all sufficiently large $k$. 
	
	Set \(S_n=CR_n=Cb_{n+1}^{-\alpha}, \; 	r_n=S_n^{1/\vartheta}\).
	Thus $r_n$ is precisely the inner scale prescribed by the definition
	of the Assouad spectrum at the outer scale $S_n$.  Since $S_n\to0$
	and $\vartheta<1$,
	\[
	\frac{R_n}{r_n}	=\frac1C S_n^{1-1/\vartheta}\longrightarrow\infty.
	\]
	
	We next compare $r_n$ with the period length $\lambda_n$. 
	\begin{align}
		\frac{r_n}{\lambda_n}
		=\frac{(Cb_{n+1}^{-\alpha})^{1/\vartheta}}
		{b_{n+1}^{-1}} =C^{1/\vartheta}
		b_{n+1}^{\,1-\alpha/\vartheta}.
	\end{align}
	Because $\vartheta\geq\alpha$ and  $C\geq1$, it follows that
	\(r_n\geq\lambda_n\) for all sufficiently large $k$.  At the endpoint $\vartheta=\alpha$,
	the two scales are comparable; when $\vartheta>\alpha$, the ratio
	$r_n/\lambda_n$ tends to infinity.
	
	We now select sufficiently many mutually separated period intervals.
	Let $I_n'$ be the closed interval having the same centre as $I_n$ and
	length $R_n/2$.  Starting from the left endpoint of $I_n'$, place
	consecutive blocks of length $4r_n$, leaving a gap of length $4r_n$
	between consecutive blocks. Since \(r_n\geq\lambda_n\), each block of length \(4r_n\) contains a
	complete \(\lambda_n\)-period of the \((n+1)\)st summand.
	Choose one complete period interval from each block and denote the
	resulting family by $\mathcal J_{n,\vartheta}$.  Since $R_n/r_n\to
	\infty$, there is a constant $c_1>0$, which may depend on
	$\vartheta$ but not on $n$, such that
	\begin{equation}\label{eq:spectrum-number-periods}
		\#\mathcal J_{n,\vartheta}
		\geq c_1\frac{R_n}{r_n}
	\end{equation}
	for all sufficiently large $k$.  The gap of length $4r_n$ between
	consecutive blocks ensures that the horizontal distance between two
	distinct intervals in $\mathcal J_{n,\vartheta}$ is at least
	$4r_n$, and hence greater than $2r_n$.
	
	Fix $J\in\mathcal J_{n,\vartheta}$.  Since $f_\alpha$ is continuous,
	$\graph(f_\alpha|_J)$ is connected.  Its vertical projection is
	therefore an interval of length  $\osc_Jf_\alpha\geq\frac{\kappa}{2}R_n$ 
	by \eqref{eq:spectrum-period-oscillation}.  On the other hand, the
	vertical projection of a set of diameter at most $r_n$ has length at
	most $r_n$.  Thus at least \({\kappa R_n}/{2r_n}\)
	sets of diameter at most $r_n$ are required to cover
	$\graph(f_\alpha|_J)$.
	
	A set of diameter at most $r_n$ cannot meet the graphs above two
	distinct intervals in $\mathcal J_{n,\vartheta}$, because those
	intervals are separated horizontally by more than $2r_n$.  Summing
	the preceding covering estimate over
	$\mathcal J_{n,\vartheta}$ and using
	\eqref{eq:spectrum-number-periods}, we obtain
	\begin{equation*} 
		N_{r_n}\bigl(\graph(f_\alpha|_{I_n})\bigr)
		\geq c_2\left(\frac{R_n}{r_n}\right)^2
	\end{equation*}
	for some constant $c_2>0$ independent of $n$.
	
	By \eqref{eq:spectrum-localization}, $\graph(f_\alpha|_{I_n})\subset B(z_n,S_n)\cap\G_{f_\alpha}(I)$. Consequently,
	\begin{align}
		N_{S_n^{1/\vartheta}}
		\bigl(B(z_n,S_n)\cap\G_{f_\alpha}(I)\bigr)\geq c_2
		\left(\frac{R_n}{r_n}\right)^2 
		=\frac{c_2}{C^2}
		\left(\frac{S_n}{S_n^{1/\vartheta}}\right)^2.
		\label{eq:spectrum-lower-covering}
	\end{align}
	
    Here $S_n\to0$, $S_n/S_n^{1/\vartheta}\to\infty$, and the
    inner scale is exactly $S_n^{1/\vartheta}$. The observation in
    \cref{sec:oscillation}, applied to
    \eqref{eq:spectrum-lower-covering}, therefore gives
    $\dimAs{\vartheta}\G_{f_\alpha}(I)\geq2$.
    The planar upper bound gives equality.
\end{proof}

\bigskip

\begin{proof}[Proof of \cref{thm:subcritical-spectrum}]
    Fix $0<\vartheta<\alpha$ with $\vartheta>(2-\alpha)/\beta_*$,
    interpreting $1/\infty=0$. Since \(f_\alpha\) is \(\alpha\)-H\"older continuous, the
	general upper bound for the Assouad spectrum of an
	\(\alpha\)-H\"older graph gives
	\begin{equation}\label{eq:subcritical-upper}
		\dim_A^\vartheta\G_{f_\alpha}(I)
		\leq
		\frac{2-\alpha-\vartheta}{1-\vartheta};
	\end{equation}
	see \cite[Theorem~1.1]{ChrontsiosGaritsisTyson2026}. We prove the
	reverse inequality. Choose a finite number $\beta'$ with
    $(2-\alpha)/\vartheta<\beta'<\beta_*$. By the definition of
    the limit superior, there is a strictly increasing sequence
    $(n_k)$ such that
    \[
     \frac{\log b_{n_k+1}}{\log b_{n_k}}>\beta'
     \qquad(k\geq1).
    \]
    We henceforth write $n=n_k$; all subsequent limits in this
    covering construction are taken along this subsequence.
	
The critical-point argument in the proof of
\cref{thm:block}\ref{item:assouad-general} applies along the full
sequence, since, by \cref{lem:dominance},  ${D_{n-1}}/{b_n^{1-\alpha}}\to 0$.
Consequently, for every sufficiently large \(n\), there is a point
\(c_n\) in a fixed compact subinterval of \(\operatorname{int}I\)
such that $P_n'(c_n)=0$.
	
	Set
	\[
	R_n=b_{n+1}^{-\vartheta},
	\qquad
	\lambda_n=b_{n+1}^{-1},
	\]
	and let \(I_n\) be the interval of length \(R_n\) centred at \(c_n\).
	Then \(I_n\subset I\) for all sufficiently large \(n\).
	By 	\eqref{eq:basic-norms}, Taylor's theorem and
	\cref{lem:dominance}, for \(x\in I_n\),
	\[
	|P_n(x)-P_n(c_n)|
	\leq
	\frac{M_2K_n}{2}|x-c_n|^2
	\leq
	\frac{M_2}{8}K_nR_n^2,
	\]
	where $K_n=(1+o(1))b_n^{2-\alpha}$. 	
	
	Moreover,
	\[
	\frac{
		\log\bigl(b_n^{2-\alpha}b_{n+1}^{-\vartheta}\bigr)
	}{
		\log b_n
	}
	=
	2-\alpha
	-
	\vartheta\frac{\log b_{n+1}}{\log b_n}
	<2-\alpha-\vartheta\beta'<0.
	\]
	Hence
	\[
	K_nR_n
	=
	(1+o(1))b_n^{2-\alpha}b_{n+1}^{-\vartheta}
	\longrightarrow0,
	\]
	and therefore
	\[
	\sup_{x\in I_n}|P_n(x)-P_n(c_n)|=o(R_n).
	\]
	Since \(\vartheta<\alpha\),  \(b_{n+1}^{-\alpha}=o(R_n)\), 	and \cref{lem:dominance} gives
	\(T_{n+1}=o(b_{n+1}^{-\alpha})=o(R_n)\).  Consequently, there is a constant \(C\geq1\), independent of \(n\),
	such that
	\begin{equation}\label{eq:subcritical-localization}
		\graph(f_\alpha|_{I_n})
		\subset B(z_n,CR_n),
		\qquad
		z_n=(c_n,f_\alpha(c_n)).
	\end{equation}
	
	Put
	\[
	S_n=CR_n,
	\qquad
	r_n=S_n^{1/\vartheta}
	=
	C^{1/\vartheta}b_{n+1}^{-1}.
	\]
	Thus \(r_n\asymp\lambda_n\), and
	\[
	\frac{S_n}{r_n}\asymp b_{n+1}^{1-\vartheta}
	\longrightarrow\infty.
	\]
	
	If \(J\subset I_n\) is a complete period interval of the
	\((n+1)\)st summand,   then the second estimate in
	\cref{lem:three-block}, together with \cref{lem:dominance}, gives
	\begin{align*}
		\osc_J f_\alpha
		\geq
		\kappa b_{n+1}^{-\alpha}
		-\frac{M_1D_n}{b_{n+1}}
		-2M_0T_{n+1}\geq
		\frac{\kappa}{2}b_{n+1}^{-\alpha}
	\end{align*}
	for all sufficiently large \(n\).
	
    In the remainder of the proof, \(c_1,c_2,\ldots\) denote positive
    constants that are independent of \(n\), although they may depend on
    \(\alpha\), \(\vartheta\), and \(\phi\). 	
    Since \(r_n\asymp\lambda_n\), the interval \(I_n\) contains a
	family \(\mathcal J_n\) of mutually \(2r_n\)-separated complete
	period intervals satisfying $\#\mathcal J_n \geq c_1{R_n}/{r_n}$.
	For each \(J\in\mathcal J_n\), 	\(\graph(f_\alpha|_J)\) is connected and has vertical projection
	of length at least
	\(\kappa b_{n+1}^{-\alpha}/2\). It therefore requires at least $c_2{b_{n+1}^{-\alpha}}/{r_n}$
	sets of diameter at most \(r_n\) to cover it. The horizontal
	separation of the intervals in \(\mathcal J_n\) ensures that no
	such covering set meets the graphs above two different intervals.
	Hence
	\begin{align}
		N_{r_n}\bigl(\graph(f_\alpha|_{I_n})\bigr)
		\geq
		c_3
		\frac{R_n}{r_n}
		\frac{b_{n+1}^{-\alpha}}{r_n}\geq
		c_4b_{n+1}^{2-\alpha-\vartheta}.
		\label{eq:subcritical-covering}
	\end{align}
	
	Let
	\[
	s_\vartheta
	=
	\frac{2-\alpha-\vartheta}{1-\vartheta}.
	\]
	Since
	\({S_n}/{r_n}\asymp b_{n+1}^{1-\vartheta}\), the estimate \eqref{eq:subcritical-covering} becomes
	\(N_{r_n}\bigl(\graph(f_\alpha|_{I_n})\bigr)\geq c_5 ({S_n}/{r_n})^{s_\vartheta}\).  Using \eqref{eq:subcritical-localization} and
	\(r_n=S_n^{1/\vartheta}\), we conclude that
	\[
	N_{S_n^{1/\vartheta}}
	\bigl(B(z_n,S_n)\cap\G_{f_\alpha}(I)\bigr)
	\geq
	c_5
	\left(
	\frac{S_n}{S_n^{1/\vartheta}}
	\right)^{s_\vartheta}.
	\]
	As \(S_n\to0\) and
	\(S_n/S_n^{1/\vartheta}\to\infty\), the definition of the
	Assouad spectrum yields
	\[
	\dim_A^\vartheta\G_{f_\alpha}(I)
	\geq
	s_\vartheta
	=
	\frac{2-\alpha-\vartheta}{1-\vartheta}.
	\]
	Together with \eqref{eq:subcritical-upper}, this proves the formula
    for $(2-\alpha)/\beta_*<\vartheta<\alpha$.

    If $\beta_*<\infty$, put $\vartheta_0=(2-\alpha)/\beta_*$.
    The Assouad spectrum is continuous on $(0,1)$ by
    \cite[Corollary~3.5]{FraserYu2018}. Letting
    $\vartheta\downarrow\vartheta_0$ in the formula just proved gives
    \[
     \dimAs{\vartheta_0}\G_{f_\alpha}(I)
     =\frac{2-\alpha-\vartheta_0}{1-\vartheta_0}.
    \]
    No assertion at $\vartheta=0$ is needed when $\beta_*=\infty$.
    Finally, $\beta_*>(2-\alpha)/\alpha$ and
    \cref{prop:log-gap}\textup{(ii)} imply the hypothesis of
    \cref{thm:power}\textup{(ii)}, which gives the formula for
    $\alpha\leq\vartheta<1$.
\end{proof}

\bigskip

\begin{proof}[Proof of \cref{cor:four}]
The equalities follow from \cref{thm:power} and \cref{thm:global}.  Only the strict inequalities require verification.

The assumption that $\liminf\limits_{n\to\infty}b_n^{2-\alpha}/b_{n+1}^{\alpha}<\infty$ forces
 $\beta_*\geq(2-\alpha)/\alpha>1$. Indeed, if $\beta_*<(2-\alpha)/\alpha$, then, by the definition of
 the limit superior, there exists \(\varepsilon>0\) such that
 \[
 \frac{\log b_{n+1}}{\log b_n}
 \leq
 \frac{2-\alpha}{\alpha}-\varepsilon
 \]
 for all sufficiently large \(n\). It would follow that
\[
 \frac{b_n^{2-\alpha}}{b_{n+1}^{\alpha}}
 \geq b_n^{\alpha\varepsilon}\longrightarrow\infty,
\]
contrary to the assumption. Therefore, $1<\beta_*<\infty$ and
\[
 1<1+\frac{1-\alpha}{1-\alpha+\alpha\beta_*}<2-\alpha<2.
\]
This proves all the strict inequalities.
\end{proof}

\section{Examples and open questions}\label{sec:questions}

The following examples illustrate the critical frequency growth, the
separation of dimensions, and the complete spectrum when the logarithmic
frequency ratios are unbounded. In each case, $I$ is any nondegenerate
compact interval and the phase sequence is arbitrary.

\begin{example}
Fix $0<\alpha<1$, let $q=(2-\alpha)/\alpha>1$, and choose $b_1>1$.
Define $b_{n+1}=b_n^q$ and $a_n=b_n^{-\alpha}$. Then
\[
 \frac{b_{n+1}}{b_n}=b_n^{q-1}\longrightarrow\infty,
 \qquad
 \frac{b_n^{2-\alpha}}{b_{n+1}^{\alpha}}=1.
\]
Thus the boundedness assumption in \cref{thm:power}\textup{(ii)}
holds. Consequently,
\[
 \dimL\G_{f_\alpha}(I)=1,\qquad
 \dimA\G_{f_\alpha}(I)=2,\qquad
 \dimAs{\vartheta}\G_{f_\alpha}(I)=2
 \quad(\alpha\leq\vartheta<1).
\]
Here $\beta_*=q$, so \cref{thm:global} also gives
\[
 \dimH\G_{f_\alpha}(I)=\ldimB\G_{f_\alpha}(I)
 =1+\frac{1-\alpha}{3-2\alpha},
 \qquad \udimB\G_{f_\alpha}(I)=2-\alpha.
\]
\end{example}

\begin{example} 
Let $b_n=2^{n!}$ and $a_n=b_n^{-\alpha}$, where $0<\alpha<1$.
Then
\[
 \frac{\log b_{n+1}}{\log b_n}=n+1\longrightarrow\infty,
 \qquad \frac{b_{n+1}}{b_n}=2^{n\,n!}\longrightarrow\infty.
\]
By \cref{thm:subcritical-spectrum},
\[
 \dimAs{\vartheta}\G_{f_\alpha}(I)
 =\min\left\{2,\frac{2-\alpha-\vartheta}{1-\vartheta}\right\}
 \qquad(0<\vartheta<1).
\]
Also, \cref{thm:power,thm:global} give
\begin{gather*}
 \dimL\G_{f_\alpha}(I)=\dimH\G_{f_\alpha}(I)
 =\ldimB\G_{f_\alpha}(I)=1,\\
 \udimB\G_{f_\alpha}(I)=2-\alpha,
 \qquad \dimA\G_{f_\alpha}(I)=2.
\end{gather*}
\end{example}

The conditions in \cref{thm:power} do not cover every sequence satisfying
$b_{n+1}/b_n\to\infty$.  For example, suppose that
\(b_n\asymp e^{n^2}\).
Then
\[
\frac{b_{n+1}}{b_n}\longrightarrow\infty,
\qquad
\frac{\log b_{n+1}}{\log b_n}\longrightarrow1.
\]
Since $\log b_n=n^2+O(1)$, we have
\begin{align*}
	\log\left(\frac{b_n}{b_{n+1}^{\alpha}}\right)
	&=(1-\alpha)n^2-2\alpha n+O(1),\\
	\log\left(
	\frac{b_n^{2-\alpha}}{b_{n+1}^{\alpha}}
	\right)
	&=2(1-\alpha)n^2-2\alpha n+O(1).
\end{align*}
Both expressions tend to infinity because $0<\alpha<1$. Hence neither of the gap conditions in \cref{thm:power} is satisfied, so
the theorem does not determine the Assouad or lower dimension in this case. Bara\'nski's formula nevertheless determines the Hausdorff and box
dimensions:
\[
\dimH\G_{f_\alpha}(I)
=\ldimB\G_{f_\alpha}(I)
=\udimB\G_{f_\alpha}(I)
=2-\alpha.
\]
 This leads to the following problem.

\begin{question}
	For the frequency sequence
	$b_n=e^{n^2}$,
	what are $\dimA\G_{f_\alpha}(I)$ and $\dimL\G_{f_\alpha}(I)$?
	More generally, suppose that ${b_{n+1}}/{b_n}\to\infty$ and  ${\log b_{n+1}}/{\log b_n}\to1$. 
	Can the Assouad and lower dimensions be determined from the asymptotic
	growth of $(b_n)$?  Can they depend on the phase sequence $(\theta_n)$
	or on the choice of the periodic function 	$\phi$?
\end{question}

We next consider the part of the Assouad spectrum that remains undetermined.
Under the boundedness condition in \cref{thm:power}\textup{(ii)},
\[
\dimAs{\vartheta}\G_{f_\alpha}(I)=2
\qquad(\alpha\leq\vartheta<1).
\]
For $0<\vartheta<\alpha$, the general upper bound for graphs of
$\alpha$-H\"older functions gives
\[
\dimAs{\vartheta}\G_{f_\alpha}(I)
\leq
\frac{2-\alpha-\vartheta}{1-\vartheta}
<2;
\]
see \cite{ChrontsiosGaritsisTyson2026}. Therefore, $\vartheta=\alpha$ is
the smallest parameter at which the spectrum reaches two.

Suppose now that
\begin{equation}\label{question-2}
\frac{2-\alpha}{\alpha}<\beta_*:=\limsup_{n\to\infty}
\frac{\log b_{n+1}}{\log b_n}
<\infty.
\end{equation}
\cref{thm:subcritical-spectrum} proves that the H\"older upper bound is
attained for $(2-\alpha)/\beta_*\leq\vartheta<\alpha$.
The left endpoint is included by continuity of the spectrum.
Thus, under \eqref{question-2}, the spectrum is known for every
$\vartheta\geq(2-\alpha)/\beta_*$ in $(0,1)$; only the smaller
parameters remain undetermined by these results. When $\beta_*=\infty$,
\cref{thm:subcritical-spectrum} determines the complete spectrum.

\begin{question}
Assuming \eqref{question-2}, determine
	\[
	\dim_A^\vartheta\G_{f_\alpha}(I)
	\qquad
	\left(
	0<\vartheta<\frac{2-\alpha}{\beta_*}
	\right).
	\]
	In particular, is the spectrum in this range determined by
	$\alpha$ and $\beta_*$, or can it also depend on the distribution of
    the logarithmic frequency ratios, the lower-order
	growth of \((b_n)\), the phase sequence \((\theta_n)\), or the
	periodic function \(\phi\)?
\end{question}

\end{document}